\documentclass[12pt,reqno]{amsart}
\usepackage{graphicx, color}
\usepackage{amssymb}
\usepackage{amsmath}
\usepackage{latexsym}
\usepackage[all]{xy}
\usepackage{kotex}

\makeatletter

\newtheorem{theorem}{Theorem}[section]
\newtheorem{definition}[theorem]{Definition}
\newtheorem{proposition}[theorem]{Proposition}

\newtheorem{lemma}[theorem]{Lemma}
\newtheorem{remark}[theorem]{Remark}

\begin{document}
\title[Kinematic $N$-expansive flows]
{\bf Kinematic $N$-expansive flows }

\author{Manseob Lee}
\address
   {Manseob Lee :  Department of Mathematics \\
    Mokwon University, Daejeon, 35349, Korea}
\email{lmsds@mokwon.ac.kr}

\author{Jumi Oh }
\address
     { Jumi Oh :  Department of Mathematics \\
      Sungkyunkwan University, Suwon, 16419, Korea.}
\email{ohjumi@skku.edu}

\author{Junmi Park }
\address
  { Junmi Park :  Department of Mathematics \\
 Chungnam National University, Daejeon, 34134, Korea.}
\email{pjmds@cnu.ac.kr}


%%%%%%%%%%%%%%%%%%%%%%%%%%%%%%%%%%%%%%%%%%%%%%%%%%%%%%%%%%%%%%%%%%%%%%%%%%%%%%
\begin{abstract}
In light of the rich results of expansiveness in the dynamics of diffeomorphisms, it is natural to consider another notions of expansiveness such as countably-expansive, measure expansive, $N$-expansive and so on. In this paper, we introduce the notion of $N$-expansiveness for flows on a $C^{\infty}$ compact connected Riemannian manifold by using the kinematic expansiveness which is extension of the $N$-expansive diffeomorphisms. And we prove that a vector field $X$ on $M$ is $C^1$ robustly kinematic $N$-expansive then $X$ satisfies quasi-Anosov. Furthermore, we consider the hyperbolicity of local dynamical systems with kinematic $N$-expansiveness.
\end{abstract}

%%%%%%%%%%%%%%%%%%%%%%%%%%%%%%%%%%%%%%%%%%%%%%%%%%%%%%%%%%%%%%%%%%%%%%%%%%%%%
\maketitle

\section{Introduction}
In the theory of dynamical systems, its main goal is the study of the global orbit structure of flows (homeomorphisms, diffeomorphisms) with emphasizing invariant properties. The notion of expansiveness for a homeomorphism on a compact metric space $X$ introduced by Utz \cite{U} is important in the qualitative study of dynamical systems.

A homeomorphism $f : X \to X$ is called {\it {expansive}} if there is $\delta>0$ such that for any distinct points $x,y\in X$ there exists $i \in \mathbb{Z}$ such that $d(f^i(x), f^i(y))>\delta.$ A system is expansive if two orbits cannot remain close to each other under the action of the system. So it is a dynamical property that plays a key role in the study of the stability of the dynamics.

After that, a variety of another expression for the expansiveness were introduced by many mathematicians such as countably-expansive, measure expansive, $N$-expansive and so on. As pointed out by Morales \cite{M}, he introduced a notion generalizing the usual concept of expansiveness which are called {\it $N$-expansive}, {\it measure expansive}. In \cite{A}, Artigue {\it et al} defined a notion of {\it countably-expansiveness}. Given $x\in X$ and $\delta>0,$ we consider the set,
$$ \Gamma_\delta^f(x) = \{y\in X : d(f^i(x), f^i(y)) \leq \delta \; {\rm for\; all}\; i \in \mathbb{Z}\},$$
it is called by {\it the dynamical $\delta$-ball} of $f$ centered at $x \in X.$ We say that $f$ is {\it {expansive}} if there is $\delta>0$ such that $\Gamma_\delta^f(x)=\{x\}$ for all $x \in X,$ $f$ is {\it {countably-expansive}} if there is $\delta>0$ such that for all $x \in X$ the set $\Gamma_\delta^f(x)$ is countable, and $f$ is {\it {$N$-expansive}} if there is $\delta>0$ such that $\Gamma_\delta^f(x)$ has at most $N$ elements.

Many results on dynamics for homeomorphisms (or diffeomorphisms) can be extended to the case of vector fields. Bowen and Walters (\cite{B}) extended expansiveness to flows. In \cite{AA}, Artigue introduced the notion of {\it kinematic expansive} for flows and he proved that every $C^1$ robust kinematic expansive vector field without singularities on a compact connected smooth Riemannian manifold is quasi-Anosov.

In this paper, we consider a new notion which is combined kinematic expansiveness and N-expansiveness for flows. And then we prove that a vector field $X$ on $M$ is $C^1$ robustly kinematic $N$-expansive then $X$ satisfies quasi-Anosov. Furthermore, we consider the hyperbolicity of chain recurrent sets and homoclinic classes with kinematic $N$-expansive property.

%================================================================================

\section{$N$-expansive flows}

Let $M$ be a $C^{\infty}$ compact connected Riemannian manifold, and let $d$ be the distance on $M$ induced from a Riemannian metric $\|\cdot\|$ on the tangent bundle $TM.$ Denote by $\mathfrak{X}^1(M)$ the set of all $C^1$ vector fields of $M$ endowed
with the $C^1$ topology. Then every $X\in\mathfrak{X}^1(M)$ generates
a $C^1$ flow $X_t:M\times \mathbb{R}\to M$, that is, a family of diffeomorphisms on $M$ such that $X_s\circ X_t=X_{s+t}$ for all $s,t \in\mathbb R$, $X_0=1_d$ and $dX_t(x)/dt|_{t=0}=X(x)$ for any $x\in M$. Here $X_t$ is called the \emph{integrated flow} of $X.$

In this paper, for $X,Y,\ldots \in \mathfrak{X}^1(M)$, we always denote the integrated
flows by $X_t, Y_t, \ldots$, respectively. For $x\in M$, let us denote the orbit $\{X_t(x): t\in\mathbb R\}$ of the flow $X_t$ (or $X$) through $x$ by $\mathcal{O}_X(x)$ if no confusion arise.
We say that a point $x\in M$ is a {\it singularity} of $X$ if $X(x)=\textbf{0}_x$; and an orbit $\mathcal{O}_X(x)$ is {\it periodic} (or {\it closed}) if it is diffeomorphic to the unit circle $S^1$.

A closed invariant set $\Lambda$ is called \emph{hyperbolic} for $X_t$ if there are constants $C>0$, $\lambda>0$ and a splitting ${T_x}M=E_x^s \oplus \langle X(x)\rangle \oplus E_x^u \ (x \in \Lambda)$ such that the tangent flow $D{X}_t : TM \to TM$ leaves invariant the continuous splitting and
$$\|{{DX_t}|_{E_x^s}}\| \leq Ce^{-\lambda t} \ {\rm and} \ \|{{DX_{-t}}|_{E_x^u}}\| \leq Ce^{-\lambda t}$$ {\rm for $x \in \Lambda$ and $t >0.$}
We say that $X \in \mathfrak{X}^1(M)$ is \emph{Anosov} if $M$ is hyperbolic for ${X_t}.$

A point $x \in M$ is called {\it nonwandering} if for any $t_1 >0$ and for any neighborhood $U$ of $x,$ there is $t \geq t_1$ such that $X_t(U) \cap U \neq \emptyset.$
The set of all nonwandering points of $X$ is denoted by $\Omega(X),$ and the set of periodic points of $X$ is denoted by P$(X).$  Clearly, $${\rm Sing}(X) \cup {\rm P}(X) \subset \Omega(X),$$
where ${\rm Sing}(X)$ denotes the set of singular points of $X.$

We say that $X\in\mathfrak{X}^1(M)$ satisfies $Axiom \ A$ if $P(X)$ is dense in $\Omega(X) \setminus {\rm Sing}(X)$ and $\Omega(X)$ is hyperbolic.
If $X$ satisfies Axiom $A$ then $\Omega(X)$ can be written as the finite disjoint union $\Omega(X)=\Lambda_1 \cup \cdots \cup \Lambda_l$ of closed invariant sets $\Lambda_i$ such that each $X|_{\Lambda_i}$ is transitive. Such a set $\Lambda_i$ is called a \emph{basic set} of $X$. A collection of basic sets $\Lambda_{i_1}, \cdots, \Lambda_{i_k}$ of $X$ is called a \emph{cycle} if for each $j=1,2,\cdots,k,$ there exists $a_j \in \Omega(X)$ such that $\alpha(a_j) \subset \Lambda_{i_j}$ and $\omega(a_j) \subset \Lambda_{i_{j+1}}\;(k+1\equiv 1).$
An Axiom $A$ $X \in \mathfrak{X}^1(M)$ is said to satisfy the \emph{no-cycle condition} if there are no cycles among the basic sets of $X.$

For any hyperbolic periodic point $x$ of $X$, the sets
\begin{align*}
W^s(x)=&\{y \in M : d(X_t(x), X_t(y))\to 0\; {\rm as}\;\; {t \to \infty}\}\;\; {\rm and} \\
\;W^u(x)=&\{y \in M : d(X_t(x), X_t(y))\to 0\; {\rm as}\;\; {t \to {-\infty}}\}
\end{align*}
are said to be the \emph{stable manifold} and \emph{unstable manifold} of $x$, respectively. Let  $X \in \mathfrak{X}^1(M)$ satisfy Axiom $A$. We say that $X$ satisfies the {\it quasi-transversality condition} if $T_x W^s (x) \cap T_x W^u (x) = \{\bf {0_x} \}$ for any $x\in M$.

Hereafter we assume that the \emph{exponential map} ${{\rm exp}_x}:{{T_x}M(1)} \to M$ is well defined for all $x \in M,$ where ${{T_x}M}(r)$ denotes the $r$-ball $\{v \in {{T_x}M} : \|v\| \leq r\}$ in ${{T_x}M}.$ Let $M_X$ be the set of regular points of $X \in \mathfrak{X}^1(M)$ ; i.e., $M_X=\{x \in M : X(x)\neq \textbf{0}_x\}.$ For any $x \in M_X,$ we let $N_x=({\rm Span} X(x))^{\bot}\subset T_xM \ \ {\rm and} \ \ {\Pi_{x,r}}={\rm exp}_x(N_x(r)),$ where $N_x(r)=N_x \cap {{T_x}M}(r)$ for $0 < r \leq 1.$

For any $x \in M_X$ and $t \in \mathbb{R},$ we take a constant $r>0$ and a $C^1$ map $\tau : \Pi_{x,r} \to \mathbb{R}$ such that $\tau(x)=t$ and $X_{\tau(y)}(y) \in \Pi_{{X_t(x)},1}$ for any $y \in \Pi_{x,r}.$ Then the \emph{Poincar\'{e} map} $f_{x,t} : \Pi_{x,{r_0}} \to \Pi_{{X_t}(x),1}$ is given by $${f_{x,t}}(y)=X_{\tau(y)}(y) \ \ {\rm for} \ \  y \in \Pi_{x,{r_0}}.$$
If $X_t(x)\neq x$ for $0<t\leq {t_0}$ and $r_0$ is sufficiently small, then $(t,y)\mapsto {X_t}(y)$ $C^1$ embeds $\{(t,y) \in {\mathbb{R}\times \Pi_{x,r}} :\; 0\leq t \leq \tau(y)\}$ for $0<r\leq {r_0}.$ The image $\{{X_t}(y) : \; y \in \Pi_{x,r}\; {\rm and} \;\; 0\leq t \leq {\tau(y)}\}$ is denoted by $F_x(X_t,r,t_0).$
For $\epsilon >0,$ let ${\mathcal{N}_\epsilon}(\Pi_{x,r})$ be the set of diffeomorphisms $\xi : \Pi_{x,r} \to \Pi_{x,r}$ such that supp$(\xi) \subset \Pi_{x,{\frac{r}{2}}}$ and $d_{C^1}(\xi,\;{1_d})<\epsilon.$ Here $d_{C^1}$ is the usual $C^1$ metric, ${1_d} : \Pi_{x,r} \to \Pi_{x,r}$ is the identity map, and supp$(\xi)$ is the closure of the set where it differs from ${1_d}.$

Let $\mathcal{N}=\bigcup_{x \in M_X}N_x$ be the normal bundle based on $M_X.$ Then we can introduce a flow (which is called a \emph{linear Poincar\'{e} flow} for $X$) on $\mathcal{N}$ by $$\Psi_t : \mathcal{N} \to \mathcal{N}, \; {\Psi_t}|_{N_x}=\pi_{N_{{X_t}(x)}}\circ {D_x}{X_t}|_{N_x},$$
where $\pi_{N_x} : {T_x}M \to {N_x}$ is the natural projection along the direction of $X(x),$ and ${D_x}{X_t}$ is the derivative map of $X_t.$ Then we can see that
$${\Psi_t}|_{N_x}=D_x{f_{x,t}} \  \  {\rm and} \  \ {f_{x,t}}\circ {{\rm exp}_x}={{\rm exp}_{X_t(x)}}\circ {\Psi_t}.$$
We say that $X \in \mathfrak{X}^1(M)$  is {\it quasi-Anosov} if $ \rm{sup}_{t \in \mathbb{R}} ||\Psi_t ({\it v})|| < \infty$ for $v \in \mathcal{N}$ then $v=\bf{0}$.

We say that a flow $X_t$ on $M$ is \emph{expansive} if for any $\epsilon >0$ there is $\delta >0$ such that if $x,y \in X$ satisfy $d(X_t(x), X_{h(t)}(y))\leq \delta$ for some $h\in \mathcal{H}$ and all $t \in \mathbb{R}$ then $y\in X_{[-\epsilon,\epsilon]}(x)$, where $\mathcal{H}$ denotes the set of continuous maps $h:\mathbb{R} \to \mathbb{R}$ with $h(0)=0.$ Such a $\delta$ is called an {\it expansive constant} of $X_t.$\\

For a flow $X_t$ on $M$ and all $x \in M,$ we denote by $\Gamma_{\delta}^{X_t}(x)$ the set $$\{y \in M : \exists h \in \mathcal{H}\; {\rm s.t.}\; d(X_{h(t)}(y), X_t(x)) \leq \delta,\;\forall t \in \mathbb{R} \}.$$

Morales (\cite{M}) introduced the notion of $N$-expansiveness for discrete type. We extend the case to the continuous type using the notion of kinematic expansive flows. So, we define the kinematic $N$-expansiveness for flows.

\begin{definition}
{\rm We say that $X\in\mathfrak{X}^1(M)$ is {\it {kinematic $N$-expansive}}, for given $N \in \mathbb{N},$ if for any $\epsilon >0$ there exists $\delta >0$ such that $$\Gamma_{\delta}(x)=\{y \in M : d(X_t(x), X_t(y))\leq \delta,\; \forall\; t \in \mathbb{R}\}$$ then $\Gamma_{\delta}(x)$ has at most $N$-elements.}
\end{definition}

First of all, we can check that kinematic $N$-expansive flow is disconnected and it has no singularity as following lemmas.

\begin{lemma}\label{disconn}
For $N \in \mathbb{N},$ if $X \in \mathfrak{X}^1(M)$ is kinematic $N$-expansive then it is totally disconnected.
\end{lemma}
\begin{proof}
Suppose that $X$ is not totally disconnected. Take $x, y \in M.$ Let $\mathcal{I}$ be a closed small arc with two end points $x$ and $y.$ For any $\delta>0,$ the length of $\mathcal{I}$ is less than $\delta,$ that is, $l(\mathcal{I}) \leq \delta.$ Let $e= \delta/2$ be an expansive constant. Then for any $z \in \mathcal{I},$ $$\Gamma_e^{X_t}(z)=\{ w \in M : d(X_t(z), X_t(w)) \leq e \;\;\forall t \in \mathbb{R}\}.$$ So, $\# \Gamma_e^{X_t}(z) >N.$ That is, $X$ is not kinematic $N$-expansive. This contradicts to complete the proof.
\end{proof}

\begin{remark}
A flow $X_t$ is totally disconnected, it means that for $x \in M,$ each orbits of $X_t(x)$ are separated apart. Therefore, if a flow $X_t$ has the kinematic $N$-expansive property then we can easily see that there exist only countably many orbits in a $\delta$-neighborhood of a flow $X_t$ for each $x \in M$ because $X_t$ is totally disconnected.
\end{remark}

\begin{lemma}\label{sing}
Let $X \in \mathfrak{X}^1(M).$ If $X$ is kinematic $N$-expansive then Sing$(X)=\emptyset.$
\end{lemma}
\begin{proof}
Suppose that there exists $\sigma \in$ Sing$(X)$ which is not isolated. Take $x \in M,$ and let $\mathcal{I}$ be a closed small arc with two endpoints $x$ and $\sigma.$ By Lemma \ref{disconn}, this is a contradiction. So, Sing$(X)=\emptyset.$
\end{proof}

A long-time goal in the theory of dynamical systems has been to describe and characterize systems exhibiting dynamical properties that are preserved under small perturbations. A fundamental problem in recent years is to study the influence of a robust dynamic property (that is, a property that holds for a given system and all $C^1$ nearby systems). In this context, it is important to consider the stability of the system which has the expansiveness from the robust point of view and one can find some results as follows.

M\~an\'e \cite{M1} proved that a robustly expansive diffeomorphism is quasi-Anosov, and Moriyasu {\it et. al.} \cite{S} proved the result for the flow case. Also there is a result for the robustly $N$-expansive diffeomorphism, that is, Lee \cite{ML} showed that for each $n\in \mathbb{N}$, if a diffeomorphisms $f$ belongs to $C^1$ interior of the set of $N$-expansive diffeomorphisms then $f$ satisfies quasi-Anosov.

However, there is no result of the $N$-expansiveness for flows yet. Therefore, in this paper, we want to extend the result of \cite{ML} to the flows using the property of robust kinematic $N$-expansiveness.\\

Let us prove this problem (Theorem 1), we first need following two lemmas.

\begin{lemma}\label{1}
Let $X \in \mathfrak{X}^1(M)$ have no singularities, $p \in \gamma \in P(X)$
$({X_T}(p)=p),$ where $\gamma$ is a periodic orbit of the flows and let $f : \Pi_{p,{r_0}} \to \Pi_p$ be the Poincar\'{e} map for some ${r_0}>0.$ Let $\mathcal{U} \subset \mathfrak{X}^1(M)$ be a $C^1$ neighborhood of $X,$ and let $0 < r \leq {r_0}$ be given. Then there are ${\delta_0}>0$ and $0< {\epsilon_0} < \frac{r}{2}$ such that for a linear isomorphism $\mathcal{O} : {N_p} \to {N_p}$ with $\|\mathcal{O}-{D_p}f\| < \delta_0,$ there is $Y \in \mathcal{U}$ satisfying
\begin{align*}
{\rm(i)}&\; Y(x)=X(x),\;\; {\rm if}\; x \notin F_p(X_t,r, \frac{T}{2}),\\
{\rm(ii)}&\; p \in \gamma \in P(Y_t),\\
{\rm(iii)}&\;g(x)= \left\{ \begin{array}{ll}
exp_p \circ \mathcal{O} \circ exp^{-1}_p(x),\;\; &\textrm{{\rm if} $x \in B[p,{\frac{\epsilon_0}{4}}] \cap \Pi_{p,r}$}\\
f(x),\;\; &\textrm{{\rm if} $x \notin B[p,{\epsilon_0}] \cap \Pi_{p,r}$},\\
\end{array} \right.
\end{align*}
 where $g : \Pi_{p,r} \to \Pi_p$ is the Poincar\'{e} map defined by $Y_t.$
\end{lemma}
\begin{proof}
See Lemma $1.3$ in \cite{N}.
\end{proof}

\begin{lemma}\label{2}
Let $X \in \mathfrak{X}^1(M)$ have no singularities. Suppose that $X_t(x) \neq x$ for $0< t \leq {t_0},$ and let $f : \Pi_{x,{r_0}} \to \Pi_{x'}$ $(x'=X_{t_0}(x))$ be the Poincar\'{e} map $({r_0}>0$ is sufficiently small.$)$ Then, for every $C^1$ neighborhood $\mathcal{U} \subset \mathfrak{X}^1(M)$ of $X$ and $0<r \leq {r_0},$ there is $\epsilon >0$ with the property that for any $\xi \in {\mathcal{N}_\epsilon}(\Pi_{x,r}),$ there exists $Y \in \mathcal{U}$ satisfying
\begin{align*}
& \left\{ \begin{array}{ll}
Y(y)=X(y), & \textrm{{\rm if} $y \notin {F_x}(X_t,r,{t_0})$}\\
f_Y(y)=f \circ \xi(y), & \textrm{{\rm if} $y \in \Pi_{x,r}.$}
\end{array} \right.
\end{align*}
Here $f_Y : \Pi_{x,r} \to \Pi_{x'}$ is the Poincar\'{e} map defined by $Y_t.$
\end{lemma}
\begin{proof}
See Remark $2$ in \cite{C}.
\end{proof}
%================================================================================

We say that $X \in \mathfrak{X}^1(M)$ is {\it $C^1$ robustly kinematic $N$-expansive} if there is a $C^1$ neighborhood $\mathcal{U}$ of $X$ such that every $Y \in \mathcal{U}$ is $N$-expansive. A vector field $X \in \mathfrak{X}^1(M)$ is called a {\it star vector field} (or {\it star flow}), denoted by $X \in \mathfrak{X}^*(M),$ if $X$ has a $C^1$-neighborhood $\mathcal{U}$ in $\mathfrak{X}^1(M)$ such that every singularity and every periodic orbit of $Y \in \mathcal{U}$ is hyperbolic. In \cite{G}, let $\mathfrak{X}^*(M)=\mathfrak{X}^1(M) \setminus {\rm Sing}(X),$ for $X \in \mathfrak{X}^*(M)$ if and only if $X$ satisfies both Axiom $A$ and no-cycle condition. Then we have the following lemma.

\begin{lemma}\label{l1}
If a vector field $X$ on $M$ is $C^1$ robustly kinematic $N$-expansive then $X \in \mathfrak{X}^*(M).$
\end{lemma}

\begin{proof}
Let $X$ be $C^1$ robustly kinematic $N$-expansive. Suppose that $X \notin \mathfrak{X}^*(M).$ Then for a $C^1$ neighborhood $\mathcal{U}$ of $X,$ there are $Y \in \mathcal{U}$ and a non-hyperbolic periodic point $p$ of $Y.$

Let $T>0$ be the period of $p,$ and let  $f : \Pi_{p,r_0}\rightarrow \Pi_p \;\; ({\rm for\;} r_0>0)$ be the Poincar\'{e} map of $Y_t$ at $p$. Since $p$ is a non-hyperbolic fixed point of $f$ there exists an eigenvalue $\lambda$ of ${D_p}f$ with $|\lambda|=1.$ Let $\delta_0 >0$ and $0< \epsilon_0 < r_0$ be given by Lemma \ref{1} for $\mathcal{U}$ and $r_0.$ Then for the linear isomorphism $\mathcal{O} : N_p \rightarrow N_p,$ there exists $Z \in \mathcal{U}$ such that
\begin{align*}
Z(x)&=Y(x), \;\; {\rm if}\; x \notin F_p(Y_t, r_0, \frac{T}{2}),\\
g(x)&= \left\{ \begin{array}{ll}
{\rm exp}_p \circ \mathcal{O} \circ {\rm exp}^{-1}_p(x), & \textrm{if $x \in B[p,{\frac{\epsilon_0}{4}}]  \cap \Pi_{p,r_0}$}\\
f(x) & \textrm{if $x \notin  B[p,{\epsilon_0}] \cap \Pi_{p,r_0}$}.\\
\end{array} \right.
\end{align*}
Here $g$ is the Poincar\'{e} map associated to $Z.$ Since the eigenvalue $\lambda$ of ${D_p}g$ is $1,$ we can take a non-zero-vector $v$ associated to $\lambda$ such that $\|v\|\leq\frac{\epsilon_0}{4}$ and ${\rm exp}_p(v) \in B[p,{\frac{\epsilon_0}{4}}].$ Then $$g({\rm exp}_p(v))={\rm exp}_p \circ {D_p}f \circ {\rm exp}^{-1}_p({\rm exp}_p(v))={\rm exp}_p(v).$$
Put $I_v=\{\eta  v : 0< \eta < {\frac{\epsilon_0}{8}}\}$ and ${\rm exp}_p(I_v)=\mathcal{I}_p$. Then $\mathcal{I}_p$ is an invariant small arc such that
$\mathcal{I}_p \subset B[p,{\epsilon_0}] \cap \Pi_{p,{r_0}}~ {\rm and}~ g(x)=x \ (x \in \mathcal{I}_p).$ So $Z_T(\mathcal{I}_p)=\mathcal{I}_p,$ where $Z_T$ is the time $T$-map of the flow $Z_t.$

Since $Z_T$ is the identity on ${\mathcal{I}_p}$, $Z_T$ is not $N$-expansive. In fact, there exists $\delta>0$ such that
\begin{align*}
\Gamma_{\delta}^{Z_T}(p)&=\{ x \in \mathcal{I}_p : d(Z_T(x), Z_T(p)) \leq \delta \} \\
&=\{ x \in \mathcal{I}_p : d(x, p) \leq \delta\}
\end{align*}
Then we obtain that $\infty < \#\mathcal{I}_p < \#\Gamma_{\delta}^{Z_T}(x_0)\; {\rm for \; all} \; x \in \mathcal{I}_p.$ This means that ${Z_T}$ is not kinematic $N$-expansive. This contradicts the fact that $Z \in \mathcal{U}.$
\end{proof}

\begin{proposition}
If a vector field $X$ on $M$ is $C^1$ robustly kinematic $N$-expansive then $X$ satisfies Axiom A.
\end{proposition}

\begin{proof}
By the Lemma \ref{l1} and reference \cite{G}, we can show that easily.
\end{proof}

\begin{lemma}\label{l2}
If a vector field $X$ on $M$ is $C^1$ robustly kinematic $N$-expansive then $X$ satisfies the quasi-transversality condition.
\end{lemma}

\begin{proof}
It is enough to show that if the flow $X_t$ of $X \in \mathfrak{X}^1(M)$ is $C^1$ robustly kinematic $N$-expansive then $X$ satisfies the quasi-transversality condition by applying Theorem (\cite{S}) and Lemma \ref{l1}. Suppose that $X$ does not satisfy the quasi-transversality condition. Then there exists $x \in M$ such that $${T_x{W^s}}(x) \cap {T_x{W^u}}(x) \neq \{{\bf 0}_x\},$$ and so we have $x \notin \Omega(X_t).$ Since $X$ satisfies the Axiom $A$, there exists a unique decomposition $\Omega(X_t)=\bigcup_{1\leq i \leq m}\Omega_i$ of
$\Omega(X_t)$ by basic sets $\Omega_i$. Then we get
\begin{equation*}
M=\bigcup_{1\leq i \leq m} W^s(\Omega_i)=\bigcup_{1\leq i \leq m} W^u(\Omega_i).
\end{equation*}
Since ${T_x{W^s}}(x) \cap {T_x{W^u}}(x) \neq \{{\bf 0}_x\},$ with a small $C^1$ perturbation of $X$ at $x$ by Lemma 2.6, we can construct $Y$ and an arc $\mathcal{L}_x$ centered at $x$ such that
\begin{center}
$\mathcal{L}_x\subset W^s(y,Y_t)\cap W^u(z,Y_t)$ and $y,z \in \Omega(Y_t).$
\end{center}
\noindent Finally we obtain $B_{\delta_1}(x)=B_{\delta_1}(x) \cap \mathcal{L}_x>N,$ and so $\Gamma_\delta^{Y_{t_1}}(x)>N.$ This implies that $Y_{t_1}$ is not kinematic $N$-expansive. The contradiction completes the proof.
\end{proof}

It is easy to see that if $X \in \mathfrak{X}^1(M)$ has no singularities and satisfies both Axiom $A$ and the quasi-transversality condition, then $X$ is quasi-Anosov by definition. Then we have the following theorem.\\

\noindent{\bf Theorem 1.}\label{main}
{\it If a vector field $X$ on $M$ is $C^1$ robustly kinematic $N$-expansive then $X$ satisfies quasi-Anosov.}\\

And we can see the following corollary. \\

\noindent{\bf Corollary.}\label{coro}
{\it If a vector field $X$ on $M$ is $C^1$ robustly kinematic expansive then $X$ satisfies quasi-Anosov.}\\

%================================================================================

\section{$N$-expansiveness for sub-dynamical systems}
\subsection{$N$-expansiveness for chain recurrent sets}
\  \

\noindent In this section, we consider the kinematic $N$-expansiveness of sub-dynamical systems, such as chain recurrent sets and homoclinic classes. First of all we shall describe some definitions.

\begin{definition}
{\rm Let $\Lambda$ be a closed $X_t$-invariant set. We say that $\Lambda$ is $C^1$ robustly kinematic $N$-expansive for $X$ if there exist $\mathcal{U}$ of $X$ and $U$ of $\Lambda$ such that for every $Y \in \mathcal{U},$ $$\Lambda_Y(U)=\bigcap\limits_{t \in \mathbb{R}} Y_t(U)$$ is kinematic $N$-expansive for $Y.$ }
\end{definition}

A sequence $\{(x_i, t_i) : x_i \in M ; t_i \geq 1; a<i<b\}(-\infty \leq a < b \leq \infty)$ is called a {\it $\delta$-pseudo orbit} if a {\it $\delta$-chain} of $X_t$ if for any $a< i < b-1,$ $d(X_{t_i}(x_i), x_{i+1}) < \delta.$ A point $x \in M$ is called {\it chain recurrent} if for any $\delta >0,$ there exists a $\delta$-pseudo orbit $\{(x_i, t_i) : 0 \leq i <n\}$ with $n>1$ such that $x_0=x$ and $d(X_{t_{n-1}}(x_{n-1}), x) < \delta.$ The set of all chain recurrent points of $X_t$ is called the {\it chain recurrent set} of $X_t,$ denoted by $\mathcal{CR}(X).$ It is easy to see that this set is closed and $X_t$-invariant.\\

\noindent{\bf Theorem 2.}\label{chain}
{\it The chain recurrent set $\mathcal{CR}(X)$ is $C^1$ robustly kinematic $N$-expansive if and only if $\mathcal{CR}(X)$ is Axiom A without cycles.}

\begin{proof}
Let $\mathcal{CR}(X)=\Lambda,$ for convenience. First of all, if $\Lambda$ is Axiom A without cycles then it is expansive, also, satisfies kinematic $N$-expansiveness. Now we prove that {\it "only"} part.

We use the fact that $X \in \mathfrak{X}^*(M)$ is equivalently Axiom A without cycles by the results of \cite{G}. It is enough to show that $X \in \mathfrak{X}^*(M).$ Suppose that $X \notin \mathfrak{X}^*(M).$ Then for a $C^1$ neighborhood $\mathcal{U}$ of $X,$ there are $Y \in \mathcal{U}$ and a non-hyperbolic periodic point $p$ of $Y.$ As the proof of Lemma \ref{l1}, we construct an arc $\mathcal{I}_p$ which is invariant and periodic for $Z_T.$

Since $\Lambda$ is $C^1$ robustly kinematic $N$-expansive, there are $\mathcal{V} \in \mathcal{U}$ and $U$ of $\Lambda$ such that for any $Z \in \mathcal{V},$ $\Lambda_Z(U)=\cap_{t \in \mathbb{R}} Z_t(U)$ is kinematic $N$-expansive. Clearly, $\mathcal{I}_p \subset \mathcal{CR}(Z) \subset {\Lambda}_Z(U).$ So, ${\Lambda}_Z(U)$ does not satisfy kinematic $N$-expansiveness which is a contradiction. Therefore, we complete the proof.
\end{proof}
%================================================================================

\subsection{$N$-expansiveness for homoclinic classes}

\  \

Homoclinic classes are natural candidates to replace the Smale's hyperbolic basic set in non-hyperbolic theory of dynamical systems. The relationship of expansiveness and hyperbolicity of homoclinic classes has been discussed in \cite{P, PP, SV, D}. So in this direction, we consider the homoclinic classes with Kinematic N-expanisve property.\\

Let $\gamma$ be a hyperbolic closed orbit $\gamma,$ the sets
\begin{align*}
W^s(\gamma)=&\;\{ x \in M : X_t(x) \to \gamma \;\;{\rm as}\;\; t \to \infty\}\;\; {\rm and} \\
\;W^u(\gamma)=&\;\{x \in M : X_t(x)\to \gamma\;\; {\rm as}\;\;{t \to {-\infty}}\}
\end{align*}
are said to be the {\it stable manifold} and {\it unstable manifold} of $\gamma,$ respectively. We say that the dimension of the stable manifold $W^s(\gamma)$ of $\gamma$ is the {\it index} of $\gamma,$ and denoted by $ind(\gamma).$ The {\it homoclinic class} of $X_t$ associated $\gamma,$ denoted by $H_X(\gamma),$ is defined as the closure of the transversal intersection of the stable and unstable manifolds of $\gamma,$ that is;
$$H_X(\gamma)=\overline{W^s(\gamma) \pitchfork W^u(\gamma)},$$ where $W^s(\gamma)$ is the stable manifold of $\gamma$ and $W^u(\gamma)$ is the unstable manifold of $\gamma.$

For two hyperbolic closed orbits $\gamma_1$ and $\gamma_2$ of $X_t,$ we say $\gamma_1$ and $\gamma_2$ are {\it homoclinically related}, denoted by $\gamma_1 \sim \gamma_2$, if $W^s(\gamma_1) \pitchfork W^u(\gamma_2) \neq \emptyset$ and $W^s(\gamma_2) \pitchfork W^u(\gamma_1) \neq \emptyset.$ When $\gamma_1$ and $\gamma_2$ are homoclinically related, their indices must be the same. By Smale's Theorem, it is well known that $$ H_X(\gamma)=\overline{\{ {\gamma}' : \gamma' \sim \gamma\}}.$$\\

\noindent{\bf Theorem 3.}\label{homo}
{\it If the homoclinic class $H_X(\gamma)$ is $C^1$ robustly kinematic $N$-expansive then $H_X(\gamma)$ is hyperbolic.}\\

Before proving the Theorem 3, we will see that a homoclinic class of kinematic $N$-expansive flow does not have singularities.

\begin{lemma}\label{hsing}
Let $\gamma$ be a hyperbolic closed orbit of $X_t.$ If the homoclinic class $H_X(\gamma)$ is kinematic $N$-expansive and $\sigma \in H_X(\gamma) \cap {\rm Sing}(X)$ then $\sigma$ is isolated.
\end{lemma}

\begin{proof}
Suppose that there exists $\sigma \in {\rm Sing}(H_X(\gamma))$ which is not isolated. Let $\epsilon >0$ be given and $\delta >0$ be the corresponding number from the definition of kinematic $N$-expansiveness. And let $x \in H_X(\gamma)$ such that $d(\sigma, x)< \epsilon.$ We can take $h(t)\equiv 0$ then $d(X_t(\sigma), X_{h(t)}(x))= d(\sigma, x)< \delta$ and so, $x \in \{X_t(\sigma) : |t|< \epsilon\}.$ This means that there exists an arc with two endpoints $x$ and $\sigma.$ By Lemma \ref{disconn}, this is a contradiction. Therefore, every singular point of $H_X(\gamma)$ is isolated.
\end{proof}

\begin{lemma}\label{periodic}
Let $\gamma$ be a hyperbolic closed orbit of $X_t.$ If the homoclinic class $H_X(\gamma)$ is $C^1$ robustly kinematic $N$-expansive then for any $\eta \in H_X(\gamma) \cap {\rm P}(X_t)$ is hyperbolic.
\end{lemma}

\begin{proof}
Let $H_X(\gamma)=\Lambda.$ Suppose that $\eta$ is not hyperbolic. Since $\Lambda$ is $C^1$ robustly kinematic $N$-expansive, there exist $\mathcal{U}$ of $X$ and $U$ of $\Lambda$ such that for every $Y \in \mathcal{U}(X),$ $\Lambda_Y(U)=\bigcap_{t \in \mathbb{R}} Y_t(U)$ is kinematic $N$-expansive. And there is a non-hyperbolic periodic point $p \in \eta \in P(Y_t)$ ($T>0$ is the period of $p$). By the proof of Lemma \ref{l1}, we obtain an arc $\mathcal{I}_p$ which is invariant and periodic for $Z_T.$

Now we can set $N_p=E_p^c \oplus E_p^s \oplus E_p^u.$ Then there exists $\epsilon >0$ such that exp$(E_p^c(\epsilon/4))=\mathcal{I} \subset U.$ Since $Z_T$ is the identity on ${\mathcal{I}_p}$, $\mathcal{I}_p \subset \Lambda_Z(U) \subset U.$ Then $\infty <  \#\mathcal{I}_p.$ This means that ${Z_T}$ is not kinematic $N$-expansive. This contradicts the fact that $Z \in \mathcal{U}.$ So we complete the proof.
\end{proof}

\begin{lemma}[Kupka-Smale \cite{MP}]
We say that a vector field $X \in \mathfrak{X}^1(M)$ is {\it Kupka-Smale} if it satisfies the following properties :
\begin{itemize}
\item[(a)] the critical elements of $X$ (the singularities and closed orbits) are hyperbolic,
\item[(b)] if $\sigma_1$ and $\sigma_2$ are critical elements of $X$ then the invariant manifolds $W^s(\sigma_1)$ and $W^u(\sigma_2)$ are transversal.
\end{itemize}
\end{lemma}

\begin{lemma}\label{index}
Let $\gamma$ be a hyperbolic closed orbit of $X_t.$ If the homoclinic class $H_X(\gamma)$ is $C^1$ robustly kinematic $N$-expansive then for any $\eta \in H_X(\gamma) \cap {\rm P}(X),$ {\rm ind} $\eta=$ {\rm ind} $\gamma.$
\end{lemma}

\begin{proof}
Assume that {\rm ind} $\eta \neq$ {\rm ind} $\gamma.$ Then there exist $p \in \gamma$ and $q \in \eta$ such that $d(p, q) < \delta,$ where $\delta$ is the $N$-expansive constant. Since $H_X(\gamma)$ be $C^1$ robustly kinematic $N$-expansive, $W^s_{\epsilon}(p) \cap W^u_{\epsilon}(q) \neq \emptyset.$ By the Kupka-Smale, $W^s_{\epsilon}(p) \pitchfork W^u_{\epsilon}(q)\neq \emptyset$ and {\rm ind} $\eta=$ {\rm ind} $\gamma.$ This is a contradiction, so we complete the proof.
\end{proof}

\begin{theorem}\label{hh}
Let $X \in \mathfrak{X}^1(M).$ The homoclinic class $H_X(\gamma)$ is $C^1$ robustly kinematic $N$-expansive if and only if $H_X(\gamma)$ is hyperbolic.
\end{theorem}

\begin{proof}
We show that $T_{H_X(\gamma)}M=E \oplus F$ satisfying $E$ is contracting and $F$ is expanding. For convenience, let
\begin{itemize}
\item $\Lambda=H_X(\gamma),$
\item $\mathcal{U}$ be a $C^1$ neighborhood of $X$,
\item $U$ is a neighborhood of $H_X(\gamma)$ satisfying $H_X(\gamma)=\cap_{t \in \mathbb{R}}X_t(U)$ is $N$- expansive.
\end{itemize}

To prove that the bundle $E$ is uniformly contracting, it is enough to show that $\lim_{t \to \infty}\; {\rm inf}\; \|D{X_t}|_{E_x}\|=0$ for any $x\in \Lambda.$ Suppose by contradiction that there is $x \in \Lambda$ such that $$\lim_{t \to \infty}\; {\rm inf}\; \|D{X_t}|_{E_x}\|>0.$$ Then there is $s_n \to \infty$ as $n \to \infty$ such that $$\lim_{s_n \to \infty}\frac{1}{s_n}\;{\rm log}\; \|DX_{s_n}|_{E_x}\| \geq 0.$$
Let $C^0(\Lambda)$ be the set of real continuous functions defined on $\Lambda$ with the $C^0$-topology, and define the sequence of continuous operators
\begin{align*}
{\Theta_n} :\; C^0(\Lambda) &\to \mathbb{R}\\
             \varphi \;&\mapsto \frac{1}{s_n} \int_0^{s_n} \varphi(X_s(x))ds.
\end{align*}
There exists a convergent subsequence of $\Theta_n,$ which we still denote by $\Theta_n,$ converging to a continuous map $\Theta_n : C^0(\Lambda) \to \mathbb{R}.$

Let $\mathcal{M}(\Lambda)$ be the space of measures with support on $\Lambda.$ By the Riesz's Theorem, there exists $\mu \in \mathcal{M}(\Lambda)$ such that
\begin{equation}
\int_{\Lambda} \varphi d\mu=\lim_{s_n \to \infty}\frac{1}{s_n}\int_0^{s_n} \varphi(X_s(x))ds=\Theta(\varphi),
\end{equation}
for every continuous map $\varphi$ defined on $\Lambda.$ It is clear that such $\mu$ is invariant by the flow $X.$ Define
\begin{align*}
&\varphi_{X} : C^0(\Lambda) \longrightarrow \mathbb{R} \;{\rm by }\\
&\varphi_{X}(p)=\partial_l({\rm log} \|D\phi_l |_{E_p}\|)_{l=0}=\lim_{l \to 0}\frac{1}{l} {\rm log} \|DX_l |_{E_p}\|.
\end{align*}
This map is continuous, and so it satisfies (1). On the other hand, for any $T \in \mathbb{R}$,
\begin{align}
\frac{1}{T} \int_0^T \varphi_{X}(X_s(p))ds=&\frac{1}{T}\int_0^T \partial_l({\rm log} \|DX_l |_{E_{X_x(p)}}\|)_{l=0}ds
                                                =&\frac{1}{T}{\rm log} \|DX_T|_{E_p}\|.
\end{align}
So we have
\begin{equation}
\int_{\Lambda}\varphi_{X}d\mu \geq 0.
\end{equation}
By the Birkhoff Ergodic Theorem, we have that
$$\int_{\Lambda} \varphi_{X} d\mu=\int_{\Lambda} \lim_{T \to \infty}\frac{1}{T} \int_0^T \varphi_{X}(X_s(y))dsd\mu(y).$$

Let $\Sigma_{X}$ be the set of strongly closed points. Since $\mu$ is invariant and Supp$(\mu) = \Lambda$,
$$\mu(\Lambda \cap ({\rm Sing}(\Lambda \cup \Sigma_{X})))=1.$$
So, $\mu(\Lambda \cap \Sigma_{X})>0.$

By the ergodic decomposition for invariant measures, we can suppose that $\mu$ is ergodic. Hence $\mu(\Lambda \cap \Sigma_{X})=1.$
Now we obtain that there exists $y \in \Lambda \cap \Sigma_{X}$ such that
\begin{equation}
\lim_{T \to \infty} \frac{1}{T} \int_0^T \varphi_{X}(X_s(y))ds \geq 0.
\end{equation}
Since $y \in \Sigma_{X}$ there are $\delta_n \to 0$ as $n \to \infty$, $\psi^n \in \mathcal{U}$ and $p_n \in {\rm Per}_{\psi^n}(\Lambda_{\psi^n}(U))$ with period $t_n$ such that $$\|\psi^n-X\| < \delta_n\;\;{\rm and}\;\; {\rm dist}(\psi^n_s(p_n), \phi_s(y)) < \delta_n \;\; {\rm for}\; 0 \leq s \leq {t_n},$$ where $\psi_s^n$ is the flow induced by $\psi^n.$ Observe that $t_n \to \infty$ as $n \to \infty.$ Otherwise if $y \in {\rm Per}(X)\cap \Lambda$ and $t_y$ is the period of $y,$ (2) and (4) imply that $DX_{t_y}|E_y$ expands.
Let $\gamma >0$ be arbitrarily small. By (4), there is $T_\gamma$ such that for $t \geq T_\gamma$
\begin{equation}
\frac{1}{t} \int_0^t \varphi_{X}(X_s(y))ds \geq \gamma.
\end{equation}
Since $t_n \to \infty$ as $n \to \infty$ we can assume that $t_n > T_\gamma$ for every $n.$ The continuity of the splitting $E \oplus F$ over $T_{\Lambda}M$ with the flow together with (5) given, for $n$ big enough, that $$\frac{1}{t_n}{\rm log} \|D\psi_{t_n}^n|_{E_{p_n}^{\psi^n}}\| \geq \gamma.$$ Thus $$\|D\psi_{t_n}^n|_{E_{p_n}^{\psi^n}}\| \geq e^{\gamma{t_n}}.$$ Taking $n$ sufficiently large and $\gamma <0$ sufficiently small, this last inequality contradicts. This completes the proof that $E$ is a uniformly contracting bundle.
\end{proof}

\noindent {\bf End of the Proof of Theorem 3.}
By Theorem \ref{hh}, we complete the proof.

\bigskip
\noindent{\bf Acknowledgement.}
First author supported by National Research Foundation of Korea(NRF) No. 20R1A2B4001892. Second author is supported by the Korea (NRF) grant funded by 2016R1D1A1B03931962. Third author is supported by Korea(NRF) No. 2017R1D1A1B03032148.
%=====================================================================================================

\end{document}